\documentclass[10pt]{amsart}
\pdfoutput=1

\usepackage{amsmath,amsthm,amssymb}
\usepackage{times}
\usepackage{graphicx}
\usepackage{wrapfig}%
\usepackage{caption}
\usepackage{subcaption,enumerate, color}
\usepackage{float}
\usepackage{multirow}
\usepackage[dvips]{epsfig}
\usepackage{lingmacros}
\usepackage{amsfonts}
\usepackage{bbm}
\usepackage{color}
\usepackage{graphicx}
\usepackage{latexsym}
\usepackage{enumerate}
\usepackage{cite}
\usepackage{tikz}
\usetikzlibrary{decorations.pathmorphing}
\usetikzlibrary{decorations.markings}
\usetikzlibrary{shapes,positioning}
\usetikzlibrary{patterns}
\usepackage[text={4.5in,7.2in},centering]{geometry}
\usepackage{enumitem}

\newcommand{\cP}{{\mathcal{P}}}

\theoremstyle{definition}
\newtheorem{defn}{Definition}

\newtheorem{exmp}{Example}
\newtheorem{theorem}{Theorem}
\newtheorem{lm}{Lemma}
\newtheorem{prop}{Proposition}

\numberwithin{figure}{section}

\title{Split domination, independence, and irredundance in Graphs}

\author{Stephen Hedetniemi, Fiona Knoll, Renu Laskar \\ Clemson University}

\begin{document}

\begin{abstract}  In 1997, Kulli and Janakiram \cite{KulliJanakiramSplit} defined the split dominating set: a dominating set $S$ of vertices in a graph $G = (V, E)$ is called {\em split dominating} if the induced subgraph $\langle V \setminus S\rangle$ is either disconnected or a $K_1$. In this paper we introduce the properties split independence and split irredundance. A set $S$ of vertices in a graph $G =(V,E)$ is called a {\em split independent set} if $S$ is independent and the induced subgraph $\langle V \setminus S \rangle$ is either disconnected or a $K_1$. A set $S$ of vertices in a graph $G = (V,E)$ is called a {\em split irredundant set} if for $u \in S$, $u$ has a private neighbor with respect to $V(S)$ and the induced subgraph $\langle V \setminus S\rangle$ is either disconnected or a $K_1$.

\end{abstract}

\maketitle
%%%%%%%%%%%%%%%%%%%%%%%%%%      This is to remove the page number from the bottom of the first page %%%%%%%%%%%%%%%%%%%%%%%
\thispagestyle{empty}

%%%%%%%%%%%%%%%%%%%%%%%%%%%%
%%%%%%%%%%%%%%%%%%%%%%%%%%%%

\begin{section}{Introduction}
Since the introduction of domination and the domination number, numerous papers have been written concerning its relationship with other graph properties, such as independence and irredundance. 

A set of vertices $D \subseteq V(G)$ is a \emph{dominating set} of graph $G$,  if for every vertex $v \in V \setminus D$ there is a vertex $u \in D$ such that $uv \in E(G)$. 
A set of vertices $S \subseteq V(G)$ is an \emph{independent set} of $G$, if the induced subgraph $\langle S \rangle$ has no edges. 
A set of vertices $S \subseteq V(G)$ is an \emph{irredundant set} if for every vertex $v \in S$, there is a vertex $u \in V(G)$ such that $u \in N[S]$ but $u \notin N[S\setminus\{v\}]$, that is, $v$ has its own private neighbor with respect to set $S$.  

Each of these properties has two associated parameters. The well-known parameters are the following:
	\begin{itemize} [leftmargin = *]
		\item[-] Domination number: $\gamma(G) = \min \{|S|: \text{$S$ is a dominating set} \}$,
		\item[-] Upper domination number:  \[\Gamma(G) = \max\{|S|: \text{$S$ is a minimal dominating set}\}\]
		\item[-] Lower independence number:  \[i(G) = \min \{|S|: \text{$S$ is a maximal independent set}\}\]
		\item[-] Upper independence number: \[\beta(G) = \max\{ |S|: \text{$S$ is an independent set} \}\]
		\item[-] Irredundant number: $ir(G) = \min\{|S|: \text{$S$ is a maximal irredundant set}\}$
		\item[-] Upper irredundant number: $IR(G) = \max\{|S|: \text{$S$ is an irredundant set}\}$
	\end{itemize}
 Relating these parameters, Cockayne et al.\cite{CockayneEtAl} in 1978 defined the domination inequality chain:
	\[ir(G) \leq \gamma(G) \leq i(G) \leq \beta(G) \leq \Gamma(G) \leq IR(G) .\]
This domination chain has given rise to many interesting research results. In this paper we connect these domination related parameters with the connectivity of graphs and investigate analogous chain relating domination and connectivity.

Two key concepts associated with these properties are the concepts hereditary and superhereditary.
 \begin{defn}
We say a property $P$ is {\em hereditary} if, for all sets $S$ that satisfy $P$, every set $S'\subseteq S$ also satisfies $P$.
\end{defn}

\begin{defn}
We say a property $P$ is {\em superhereditary} if, for all sets $S$ that satisfy $P$, every set $S'\supseteq S$ also satisfies $P$.
\end{defn}
In addition,
\begin{defn}
 Let $S$ be a set satisfying a property $\cP$.
\begin{itemize}
	\item The set $S$ is \emph{minimal} with respect to $\cP$ if no subset $S' \subseteq S$ satisfies $\cP$.
	\item The set $S$ is \emph{maximal} with respect to $\cP$ if no superset $S'\supseteq S$ satisfies $\cP$.
	\item The set $S$ is \emph{1-minimal} with respect to $\cP$ if for any $v \in S$, $S\setminus\{v\}$ does not satisfy $\cP$.
	\item The set $S$ is \emph{1-maximal} with respect to $\cP$ if for any $v \notin S$, $S \cup \{v\}$ does not satisfy $\cP$.	
\end{itemize}
\end{defn}
In general, 1-minimal property is not necessarily a minimal property on a set $S$ and 1-maximal property is not  necessarily a maximal property on a set $S$. However, the following are well known  \cite{HedetBook}.
\begin{prop}
If the property $\cP$ on a set $S$  is superhereditary, then $\cP$ is minimal if and only if $\cP$ is 1-minimal.
\end{prop}
\begin{prop}
If the property $\cP$ on a set $S$  is hereditary, then $\cP$ is maximal if and only if $\cP$ is 1-maximal.
\end{prop}

In 1997, Kulli and Janakiram \cite{KulliJanakiramSplit} considered the relationship between the property domination and the property connectivity, which is neither hereditary nor superhereditary. 
\begin{defn}
A \emph{vertex cut set} is a set of vertices $S \subseteq V(G)$ such that $G\setminus S$ is either a $K_1$ or a disconnected graph. The \emph{connectivity} $k(G)$ of a graph $G$ is the size of the smallest vertex cut set of graph $G$.
\end{defn}

Kulli and Janakiram combined the concepts of domination and connectivity to form the definition of a split dominating set. Given a graph $G$ with no isolate vertices, a dominating set $S \subseteq V(G)$ is a \emph{split dominating set} if the induced subgraph $\langle V \setminus S \rangle$ is either disconnected or a $K_1$. In 2000, Kulli and Jankariam continued their work and introduced \emph{nonsplit domination}  \cite{KulliJanakiramNonsplit} and in 2005, \emph{strong split domination}  \cite{KulliJanakiramStrong}. Furthering this work, Chelvam and Chellathurai \cite{ChelvamChellathurai} found bounds on the parameters of both split and nonsplit domination and discovered the relationship between these parameters. In 2010, Bibi and Selvakumar \cite{BibiSelvakumar} combined the inverse dominating set and a split domination set, resulting in the \emph{inverse split domination set}. Given a minimal dominating set $D$, if $V\setminus D$ contains a dominating set $D'$, $D'$ is called an \emph{inverse split dominating set}. In this paper, we combine the concept of connectivity with the properties independence and irredundance.

\end{section}

\begin{section}{Split Domination}
For the rest of this paper, let $G$ be a finite, undirected, connected graph that does not contain loops or multiple edges.

\begin{defn}{(Kulli and Janakiram \cite{KulliJanakiramSplit})} \label{def: split_set}
A dominating set $S \subseteq V(G)$ is a \emph{split dominating set} if the induced subgraph $\langle V \setminus S \rangle$ is either disconnected or a $K_1$.
\end{defn}

\begin{defn} \cite{KulliJanakiramSplit} \label{def: minimal split_set}
	A split dominating set $S$ is a \emph{minimal split dominating set} if
		\begin{enumerate}
			\item every vertex $v\in S$ has a private neighbor with respect to $S$ or
			\item for every vertex $v \in S$, the induced graph $\langle (V\setminus S) \cup \{u\} \rangle$ is connected.
		\end{enumerate}
\end{defn}

\begin{defn} \cite{KulliJanakiramSplit} Let $G = (V,E)$ be a graph. Then
	\begin{itemize}
		\item $\gamma_{s}(G) = \min \{|S|: \text{S is a split dominating set}\}$ is the \emph{split domination number}, and
		\item $\Gamma_{s}(G) = \max \{|S|: \text{S is a minimal split dominating set}\}$ is the \emph{upper split domination number}.
	\end{itemize}
\end{defn}

Some known inequalities and bounds that are of interest are the following:
\begin{itemize}
	\item \cite{KulliJanakiramSplit} $\gamma (G)  \leq \gamma_s (G)$ 
	\item \cite{KulliJanakiramSplit} $k(G) \leq \gamma_s(G)$ 
	\item \cite{KulliJanakiramSplit} $\gamma_s (G) \leq n \cdot \Delta(G)/ (\Delta(G) +1)$ where $|V| =n$ and $\Delta(G)$ is the maximum degree of $G$
	\item \cite{KulliJanakiramSplit} When $\text{diam}(G) = 2$, $\gamma_s (G) \leq \delta (G)$ where $\delta(G)$ is the minimum degree of $G$ 
	\item \cite{ChelvamChellathurai} $\gamma(G) + \gamma_s (G) \leq n$ 
\end{itemize}

\end{section}

\begin{section}{Independence and Split Domination}
In this section, we combine the properties independence and connectivity.

\begin{defn} \label{def: indep split_set}
A set $S$ is a \emph{split independent set} if
	\begin{enumerate}
		\item $S$ is independent and 
		\item the induced graph $\langle V \setminus S \rangle$ is either disconnected or a $K_1$.
	\end{enumerate}
\end{defn}

First, we note that a split independent set is not necessarily split dominating.
Consider a path $P_n$, $n \geq 4$. This path can be disconnected by the removal of one vertex, but that one vertex is not a dominating set.
We also note that not every graph contains a split independent set. For example, a complete graph $K_n$ and a wheel $W_n$, $n\geq 3$ do not contain a split independent set. More generally,

\begin{prop} 
A 2-tree graph $G=(V,E)$ does not contain a split independent set.
\end{prop}
\begin{proof}
 Let $G=(V,E)$ be a  2-tree graph. Then in order to disconnect the graph, one must at least disconnect a $K_3$. To disconnect a $K_3$, that is to have a $K_1$, two vertices must be removed; however, these vertices are adjacent.
\end{proof}

\begin{defn} \label{def: maxindep split_set}
A \emph{maximal split independent} set $S$ is a split independent set such that for any $v \in V \setminus S$ at least one of the following is true:
	\begin{enumerate}
		\item $S \cup \{v\}$ is not independent, i.e. $S$ is an independent dominating set
		\item the induced graph $\langle V-\{S \cup \{v\}\} \rangle$ is connected.
	\end{enumerate}
\end{defn}

\begin{defn} Let $G=(V,E)$ be a graph. Then
	\begin{itemize}
		\item $i_s(G) = \min \{|S|: \text{S is a maximal split independent set}\}$ is the \\ \emph{lower split independent number}, and
		\item $\beta_s(G) = \max \{|S|: \text{S is a maximal split independent set}\}$ is the \emph{split independent number}.
	\end{itemize}
\end{defn}

\begin{lm} \label{lemma: maxindep_minsplit}
If a split independent set exists for graph $G=(V,E)$, then a maximal split independent set $S$ for $G$ is a minimal split dominating set.

\end{lm}

\begin{proof}
Let $S$ be a maximal split independent set. We  will first show that $S$ is a dominating set. Clearly,  $v \in S$ is dominated, so we will consider $v \in V \setminus S$. Suppose $v$ is not dominated by $S$, that is $v \notin N[S]$. Then $\{v\} \cup S$ is an independent set, but it cannot be a vertex-cut set; otherwise, it would contradict the maximality of $S$. So, $\langle V \setminus \{\{v\} \cup S\} \rangle$ is connected. We have two cases to consider
\begin{itemize}
	\item Case 1: $v$ is an isolate.  In our scenario, $G$ has no isolates.
	\item Case 2: $v$ is not an isolate. If $v$ is not an isolate, then $v$ lies in the neighborhood of a vertex $u$ not in $S$. Hence, eliminating $v$ in addition to the set $S$ from the set of vertices would not create a connected graph, that is $\langle V \setminus \{\{v\} \cup S\} \rangle$ is disconnected. As $\{v\} \cup S$  is independent and a vertex-cut set, $\{v\} \cup S$  is a split indpendent set, which contradicts the maximality of $S$.
\end{itemize}
So, $v$ must be dominated by $S$ and as a result, $S$ is a dominating set, more specifically a split dominating set.

Now we want to show that $S$ is a minimal split dominating set. Suppose $S$ is not minimal split dominating set. Then there is a vertex $v \in S$ such that $S \setminus \{v\}$ is split dominating and hence $v$ is dominated by $S$ which is a contradiction to the independence of $S$.
\end{proof}

\begin{prop} \label{prop:maxindep_minsplit_param} For a graph G, 
	\[\beta_s(G) \leq \Gamma_s(G) \quad \text{and} \quad \gamma_s(G) \leq i_s(G) \]
\end{prop}
\begin{proof}
It follows from Lemma \ref{lemma: maxindep_minsplit}.
\end{proof}

\end{section}

%%%%%%%%%%%%%%%%%%%%%%%%%%%%%%%%%%%%%%%
\begin{section}{Irredundance and Split Domination}

We now consider the property of irredundance with the concept of connectivity.

\begin{defn} \label{def: irred split_set}
A \emph{split irredundant set} is a set $S$ such that 
	\begin{enumerate}
		\item for $u \in S$, $u$ has a private neighbor with respect to $V(S)$ and
		\item the induced graph $\langle V \setminus S \rangle$ is either disconnected or a $K_1$.
	\end{enumerate}
\end{defn}

Similar to a split independent set, a split irredundant set is not necessarily a split dominating set nor does it exist in every situation. In Figure \ref{fig:IrredNotDom}, the set $\{u_1, u_2, u_3\}$ is a split irredundant set, but it is not a split dominating set. As with split independent sets, a split irredundant set does not exist for wheels $W_n$ and complete graphs $K_n$, $n \geq 3$.

\begin{figure}[h]
\begin{center}
\begin{tikzpicture}[every loop/.style={}]
\draw (0,0) circle (.3) node (A) {$u_1$}
(2,0) circle (.3) node (B) {$u_2$}
(4,0) circle (.3) node (C) {$u_3$}
(2/3,1.5) circle (.3) node (D) {$v_1$}
(4/3,1.5) circle (.3) node (E) {$v_2$}
(8/3,1.5) circle (.3) node (F) {$v_3$}
(10/3,1.5) circle (.3) node (G) {$v_4$}
(1,3.1) circle (.3) node (H) {$w_1$}
(3,3.1) circle (.3) node (J) {$w_2$};
\draw[line width=.5pt] (A) edge (B);
\draw[line width=.5pt] (B) edge (C);
\draw[line width=.5pt] (A) edge (D);
\draw[line width=.5pt] (B) edge (E);
\draw[line width=.5pt] (D) edge (H);
\draw[line width=.5pt] (E) edge (H);
\draw[line width=.5pt] (B) edge (F);
\draw[line width=.5pt] (C) edge (G);
\draw[line width=.5pt] (F) edge (J);
\draw[line width=.5pt] (G) edge (J);

\end{tikzpicture}
\end{center}
\caption{In the figure above, $\{u_1,u_2,u_3\}$ is a maximal split irredundant set; however, it is not a split dominating set. \label{fig:IrredNotDom}}
\end{figure}
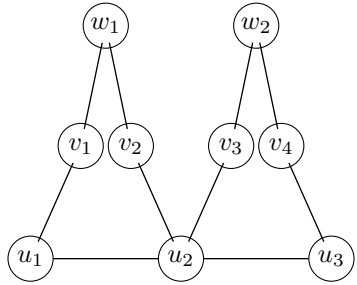

\begin{defn} \label{def: max irred split_set}
A \emph{maximal split irredundant set} $S$ is a split irredundant set such that for every $v \in V \setminus S$ one of the following holds true:
	\begin{enumerate}
		\item $v$ does not have a private neighbor with respect to $V(S \cup \{v\})$  or
		\item the induced graph $\langle V \setminus \{S \cup v\} \rangle $ is connected
	\end{enumerate}
\end{defn}

\begin{defn} Let $G=(V,E)$ be a graph. Then
	\begin{itemize}
		\item $ir_s(G) = \min \{|S|: \text{S is a maximal split irredundant set}\}$ is the \emph{lower split irredundance number}, and
		\item $\mbox{IR}_s(G) = \max \{|S|: \text{S is a maximal split irredundant set}\}$ is the \emph{upper split irredundance number} 
	\end{itemize}
\end{defn}	
	
\begin{lm} \label{lemma: minsplit_maxirred}
If a split irredundant set is defined for $G$, then a minimal split dominating set $S$ is a maximal split irredundant set.
\end{lm}

\begin{proof}
	Suppose $S$ is a minimal split dominating set. Then for $u\in S$ either
		\begin{enumerate}
			\item $S \setminus \{u\}$ is not dominating or
			\item the induced graph $\langle (V \setminus S) \cup \{u\}\rangle$ is connected.
		\end{enumerate}
Clearly, $S$ is then split irredundant.\\
We now want to show $S$ is a maximal split irredundant set, that is, for $v \in V\setminus S$ 
	\begin{enumerate}
		\item $v$ does not have a private neighbor with respect to $V(S \cup \{v\})$  or
		\item $\langle V \setminus \{S \cup v\}\rangle$ is connected.
 	\end{enumerate}			
Let $v$ be an arbitrary vertex in $V \setminus S$. Since $S$ is a minimal split dominating set, $v \in N[S]$ and hence, $S \cup \{v\}$ is not irredundant.
As a result, $S$ is a maximal split irredundant set. 	
\end{proof}

\begin{prop} \label{prop:minsplit_maxirred_param}
For a graph $G$,
	\[\Gamma_s(G) \leq \mbox{IR}_s(G) \quad \text{and} \quad ir_s(G) \leq \gamma_s(G)\]
\end{prop}
\begin{proof}
	It follows from Lemma \ref{lemma: minsplit_maxirred}.
\end{proof}

\end{section}

%%%%%%%%%%%%%%%%%%%%%%%%%%%%%%%%%%%%%%%%%
\begin{section}{Analog to the Domination Chain and Properties of Parameters}

Analogous to the domination chain, we have the following:

\begin{theorem} The split dominating chain holds for graph $G$ when both the split independent set and split irredundant set exist:
\[ir_s(G) \leq \gamma_s(G) \leq i_s(G) \leq \beta_s(G) \leq \Gamma_s (G) \leq \mbox{IR}_s(G)\]
\end{theorem}

\begin{proof}
Follows from Propositions \ref{prop:maxindep_minsplit_param} and \ref{prop:minsplit_maxirred_param}.
\end{proof}

These bounds are tight. For example,
\begin{exmp} For the bipartite graph $K_{m,n}$, 
	\[ir_s(G) = \gamma_s(G) = i_s(G) = \min\{m,n\}\]
and
	\[\beta_s(G) = \Gamma_s (G) = \mbox{IR}_s(G) = \max\{m,n\}.\]
When $m=n$, you obtain equality between all of the parameters.
\end{exmp}

\begin{exmp} For a path $P_n$,
	\[ir_s(G) = \gamma_s(G) = i_s(G) = \left\lceil \frac{n}{3} \right\rceil\]
and
	\[\beta_s(G) = \Gamma_s (G) = \mbox{IR}_s(G) = \left\lceil \frac{n}{2} \right\rceil.\]
\end{exmp}

\begin{exmp} For a cycle $C_n$,
	\[ir_s(G) = \gamma_s(G) = i_s(G) = \left\lceil \frac{n}{3} \right\rceil\]
and
	\[\beta_s(G) = \Gamma_s (G) = \mbox{IR}_s(G) = \left\lfloor \frac{n}{2} \right\rfloor.\]
\end{exmp}

The connectivity of a graph, $k(G)$, provides us with a lower bound for all of the parameters:
\begin{prop} For $G=(V,E)$,
	\[k(G) \leq ir_s(G) \quad k(G) \leq \gamma_s(G) \quad k(G) \leq i_s(G)\] 
\end{prop}
\begin{proof}
	Follows from the definitions.
\end{proof}

\end{section}

%%%%%%%%%%%%%%%%%%%%%%%%%%%%%%%%%%%%%
\begin{section}{Nonsplit Domination, Irredundance, and Independence}

In 2000, Kulli and Janakiram considered the opposite side of the spectrum, nonsplit domination with its related parameters.

\begin{defn}(Kulli and Janakiram \cite{KulliJanakiramNonsplit}) \label{def: nonsplit_set}
A  dominating set $S \subseteq V(G)$ is a \emph{nonsplit dominating set} if the induced subgraph $\langle V \setminus S \rangle$ is connected.
\end{defn}

\begin{defn} \cite{KulliJanakiramNonsplit} \label{def: minimal nonsplit_set}
	A nonsplit dominating set $S$ is a \emph{minimal nonsplit dominating set} if
		\begin{enumerate}
			\item every vertex $v$ has a private neighbor with respect to $S$ or
			\item for $v\in S$, the induced graph $\langle (V\setminus S) \cup \{u\} \rangle$ is disconnected or a $K_1$.
		\end{enumerate}	
\end{defn}

\begin{defn} Let $G = (V,E)$ be a graph. Then
	\begin{itemize}
		\item $\gamma_{ns}(G) = \min \{|S|: \text{S is a nonsplit dominating set}\}$ is the \emph{nonsplit domination number}, and
		\item $\Gamma_{ns}(G) = \max \{|S|: \text{S is a minimal nonsplit dominating set}\}$ is the \emph{upper nonsplit domination number}.
	\end{itemize}
\end{defn}

We will briefly consider the concept of nonsplit with that of independence and irredundance and the associated parameters.

\begin{defn} \label{def: indep nonsplit_set}
A set $S$ is a \emph{nonsplit independent set} if
	\begin{enumerate}
		\item $S$ is independent and 
		\item the induced graph $\langle V \setminus S \rangle$ is connected.
	\end{enumerate}
\end{defn}

\begin{defn} \label{def: maxindep nonsplit_set}
A \emph{maximal nonsplit independent} set $S$ is a nonsplit independent set such that for any $v \in V \setminus S$  one of the following is true
	\begin{enumerate}
		\item $S \cup \{v\}$ is not independent, i.e. $S$ is an independent dominating set
		\item the induced graph $\langle V-\{S \cup \{v\}\} \rangle$ is disconnected or a $K_1$.
	\end{enumerate}	
\end{defn}

\begin{defn} Let $G=(V,E)$ be a graph. Then
	\begin{itemize}
		\item $i_{ns}(G) = \min \{|S|: \text{S is a maximal nonsplit independent set}\}$ is the\\  \emph{lower nonsplit independent number}, and
		\item $\beta_{ns}(G) = \max \{|S|: \text{S is a maximal nonsplit independent set}\}$ is the \emph{nonsplit independent number}.
	\end{itemize}
\end{defn}

\begin{defn} \label{def: irred nonsplit_set}
A \emph{nonsplit irredundant set} is a set $S$ such that 
	\begin{enumerate}
		\item for $u \in S$, $u$ has a private neighbor with respect to $V(S)$ and
		\item the induced graph $\langle V \setminus S \rangle$ is connected.
	\end{enumerate}
\end{defn}

\begin{defn} \label{def: max irred nonsplit_set}
A \emph{maximal nonsplit irredundant set} $S$ is a nonsplit irredundant set such that for every $v \in V \setminus S$ one of the following holds true
	\begin{enumerate}
		\item $ v$ does not have a private neighbor with respect to $V(S \cup \{v\})$  \textbf{or}
		\item the induced graph $\langle V \setminus \{S \cup v\} \rangle $ is disconnected or a $K_1$.
	\end{enumerate}	
\end{defn}

\begin{defn} Let $G=(V,E)$ be a graph. Then
	\begin{itemize}
		\item $ir_{ns}(G) = \min \{|S|: \text{S is a maximal nonsplit irredundant set}\}$ is the\\ \emph{lower nonsplit irredundance number}, and
		\item $\mbox{IR}_{ns}(G) = \max \{|S|: \text{S is a maximal nonsplit irredundant set}\}$ is the \emph{upper nonsplit irredundance number}. 
	\end{itemize}
\end{defn}	

We note that there is no direct relationship between the parameters of nonsplit domination, irredundance, and independence for any generic graph.

\begin{exmp}
For a path $P_n$, 
	\[\gamma_{ns} (P_n) = \Gamma_{ns}(P_n) =  n-2\]
	\[i_{ns}(P_n) = \beta_{ns}(P_n) = 2\]
	\[ir_{ns}(P_n) = IR_{ns}(P_n) = 2\] 
In the case of a path $P_n$, $n \geq 5$, the maximal nonsplit independent set is not a nonsplit dominating set.\\
For a cycle $C_n$,
	\[\gamma_{ns} (C_n) = \Gamma_{ns}(C_n) = n-2\]
	\[i_{ns}(P_n) = \beta_{ns}(P_n) = 1\]
	\[ir_{ns}(P_n) = IR_{ns}(P_n) = 2\] 
In the case of a cycle $C_n$, $n \geq 4$, the maximal nonsplit independent set is not a nonsplit dominating set.\\
For a wheel $W_n$,
	\[\gamma_{ns} (W_n) = 1; \quad \Gamma_{ns} (W_n) = \Gamma(C_n) \]
	\[i_{ns}(W_n) = 1;  \quad \beta_{ns}(W_n) = \beta(C_n)\]
	\[ir_{ns}(W_n) = 1; \quad IR_{ns}(W_n) = IR(C_n)\]	
For a complete bipartite graph $K_{m,n}$, $2\leq m\leq n$,
	\[ \gamma_{ns}(K_{m,n})= \Gamma_{ns}(K_{m,n})=2\]
	\[ i_{ns}(K_{m,n}) = m-1, \quad \beta_{ns}(K_{m,n}) = n-1\]
	\[ir_{ns}(K_{m,n})  = 2, \quad IR_{ns}(K_{m,n}) =n-1\]
\end{exmp}

\section{Open Problems}
In the future, we would like to consider the relationship between the nonsplit parameters, the split parameters and the original parameters, $i, \beta, \gamma, \Gamma, ir,$ and $Ir$. In addition, we would like to find better bounds for the parameters introduced in this paper. 

\end{section}

%\bibliographystyle{plain} 
%\bibliography{VertexCutBib}{}

\end{document}